\documentclass[reqno,11pt]{amsart}
\usepackage{amssymb,amsthm}

\def\@begintheorem#1#2{\par\bgroup{\bf #1\ \bf {#2}. }\it\ignorespaces}
\def\@opargbegintheorem#1#2#3{\par\bgroup{\bf #1\ \bf #2\ (#3). }\it\ignorespaces}
\def\@endtheorem{\egroup}

\newtheorem{theorem}{Theorem}[section]
\newtheorem{lemma}[theorem]{Lemma}

\newtheorem{remark}[theorem]{Remark}

\newcommand{\fer}[1]{{(\ref{#1})}}
\def\RR{\mathbb{R}}

\def\R{\mathbb{R}}

\def\RR{\hbox{{\tiny \rm I}\kern-.1em\hbox{{\tiny \rm R}}}}

\def\NN{\hbox{I\kern-.2em\hbox{N}}}

\newcommand{\mF}{\mathcal F}
\newcommand{\mR}{\mathcal R}
\newcommand{\mN}{\mathcal N}
\newcommand{\mH}{\mathcal H}
\newcommand{\mI}{\mathcal I}
\newcommand{\mRH}{\mathcal R\mathcal H}

\newcommand{\be}{\begin{equation}}
\newcommand{\ee}{\end{equation}}
\newcommand{\bee}{\begin{equation*}}
\newcommand{\eee}{\end{equation*}}
\newcommand{\beqa}{\begin{eqnarray}}
\newcommand{\eeqa}{\end{eqnarray}}
\newcommand{\beqao}{\begin{eqnarray*}}
\newcommand{\eeqao}{\end{eqnarray*}}

\title[Renyi entropy and nonlinear diffusions]{Renyi entropy and improved equilibration rates to self-similarity for nonlinear diffusion equations}

\author{J. A. Carrillo}
\address{Department of Mathematics, Imperial College London,
South Kensington Campus, London SW7 2AZ, UK.}
\email{carrillo@imperial.ac.uk}

\author{G. Toscani}
\address{Department of Mathematics,
University of Pavia, via Ferrata 1, 27100 Pavia, ITALY.}
\email{giuseppe.toscani@unipv.it}

\begin{document}

\hyphenation{bounda-ry rea-so-na-ble be-ha-vior pro-per-ties cha-rac-te-ris-tic}

\maketitle

\begin{abstract}
We investigate the large-time asymptotics of nonlinear diffusion equations
$u_t = \Delta u^p$ in dimension $n \ge 1$, in the exponent interval $p >
n/(n+2)$, when the initial datum $u_0$ is of bounded second moment.
Precise rates of convergence to the Barenblatt profile in terms of the
relative R\'enyi entropy are demonstrated for finite-mass solutions defined
in the whole space when they are re-normalized at each time $t> 0$ with
respect to their own second moment, as proposed in \cite{Tos05, CDT05}.
The analysis shows that the relative R\'enyi entropy exhibits a better
decay, for intermediate times, with respect to the standard Ralston-Newton
entropy. The result follows by a suitable use
of the so-called concavity of R\'enyi entropy power, recently proven in
\cite{ST}.
\end{abstract}

\vskip 3mm

{\small \noindent {\bf 2000 AMS subject classification.} 35K55, 35K60, 35K65, 35B40.

\


\


\section{Introduction}

In this work, we study the large-time behavior of the positive
solutions $v(x,t)$ of the nonlinear diffusion equation
\begin{equation}
{\partial v \over \partial t} = \Delta v^p ,\qquad (x\in \RR^n,
t>0), \label{poro}
\end{equation}
with initial data
\begin{equation*}
v(x,t=0) = v_0(x) \geq 0, \qquad (x \in \RR^n) \label{por2}.
\end{equation*}
Our analysis will be restricted to initial data with finite kinetic energy
(second moment), and it will include both the case $p>1$, (Porous
medium equation (PME)), and the case $p<1$ (Fast diffusion (FD)). We
work in dimensions $n\ge 1$, and the range of exponents in the fast
diffusion case is $\bar p < p < 1$ with $\bar p = n/(n+2)$, which is
a part of the so-called fast diffusion range $p < 1$. The particular
subinterval of $p$ is motivated by the fact that, in this range of the parameter,
the so-called Barenblatt solution \cite{Ba1, Bar, ZK}, which
serves as the model for the asymptotic behavior of a wide class of
solutions, has finite second moment. In the fast diffusion case,
the Barenblatt solution is given by
\begin{equation}\label{bar}
\mathcal{B}_M (x,t) = t^{-n/\lambda} B(|x|t^{-1/\lambda})= \left( \frac{Ct}{|x|^2 + A t^{2/\lambda}}
\right)^{1/(1-p)},
\end{equation}
where $\lambda = 2- n(1-p)$  and $C= 2p\lambda/(1-p)$ are fixed
positive constants. The polynomial decay of $B(r)$ as $r \to \infty$ establishes a relation between the parameter $p$ and the number of its moments which remain bounded. In case $p > 1-2/n$, the Barenblatt solution is integrable, and in this case the  constant $A>0$ is fixed
by mass conservation
$$
\int_{\RR^n} B(|x|)\,dx = M , \quad t \ge 0\ .
$$
Then the second moment of $\mathcal{B}_M (x,t)$ remains bounded, so that
$$
\int_{\RR^n}|x|^2 \mathcal{B}_M (x,t)\,dx = t^{-2/\lambda}\int_{\RR^n} |x|^2 B(|x|)\,dx, \quad t >0
$$
whenever $p > \bar p$. The solution to equation \fer{poro} satisfies mass and momentum
conservations, so that
\be\label{cons}
  \int_{\RR^n} v(x,t)\, dx =1 \, ; \quad    \int_{\RR^n} x \, v(x,t)\, dx =  0 \,;  \quad t \ge 0.
\ee
Hence, without loss of generality, one can always assume that
$v_0(x)$ is a probability density of first moment equal to zero.

The conservation properties of the solution play an important role
in connection with the rate of decay towards the self-similar
Barenblatt solution. This question, motivated by the search of
improved convergence rates when moment conditions are imposed in the
framework of Wasserstein,  has been addressed in \cite{CDT05,DMC1,DMC2}.
As far as the conservation of first
moment (position of center of mass) is concerned, the question of
improved rates has been precisely formulated in \cite{CDT05}, and
solved in \cite{BDGV,DMCK} in the weighted $L^2$-framework. Indeed, for
initial data satisfying \fer{cons}, it has been found in
\cite{CDT05} that the rate of convergence of solutions to the
one-dimensional porous medium equation towards a self-similar
profile improves by considering Barenblatt solutions satisfying
\fer{cons}.

These results suggested that an improved rate of decay towards
Barenblatt solutions could be obtained only by choosing moments in a
suitable way. However, since mass and momentum are the only
conserved quantities, further improvements which involve
conservation of higher order moments could be achieved only asymptotically
\cite{BDGV,DMCK} due to linearizations. Here, we explore how to get this improvement
at an initial stage based on scaling arguments.

Let us define the energy $E(v(t))$ of a solution $v(x,t)$ as its
second moment
\bee\label{temp}
 E(v(t)) =  \int_{\RR^n}|x|^2v(x,t)\,dx\, .
\eee
Then, $E(v(t))$ increases in time from $E_0 = E(v_0)$, and its
evolution is given by the nonlinear law
\be\label{ene}
 \frac{dE(v(t))}{dt} = 2n \int_{\RR^n}v^p(x,t)\, dx \ge 0,
\ee
which is not explicitly integrable unless $p=1$. The monotonicity in time of the energy suggests to change variables with the ansatz
 \be\label{cv}
v(x,t) = \left( \frac{E(v(t))}{E_0}\right)^{-n/2} u\left[x \left(
\frac{E(v(t))}{E_0}\right)^{-1/2}, \tau(t) \right],
 \ee
where
 \be\label{new-t}
\tau(t) = \frac{E_o}{2n} \log \left(\frac{E(v(t))}{E_0} \right), \quad
\tau(0) = 0.
 \ee
With this scaling,
 \[
E(v(t)) = \int_{\RR^n} |x|^2 \, v(x,t) \, dx =
\frac{E(v(t))}{E_0}\int_{\RR^n} |y|^2 \, u(y,\tau) \, dy.
 \]
Therefore, for any time $\tau >0$
 \bee\label{new-ev}
\int_{\RR^n} |y|^2 \, u(y,\tau) \, dy = E_0 = \int_{\RR^n} |x|^2
\, v_0(x) \, dx,
 \eee
and the scaled function $u(y,\tau)$,  where $ y = ({E(v(t))}/{E_0})^{1/2} x$,
maintains its second moment constant in time.

The nonlinear scaling  \fer{cv} has a robust meaning. As outlined in
\cite{CV07},  the second moment corresponds to the kinetic energy of
the probability distribution in kinetic theory, where $x$ represents
speed, $v$ is the particle density and $E(v)$ is the kinetic
energy. In fact, solutions where the only dependence on time is
through their temperature are of asymptotic importance in several
kinetic models for which related scalings have been used, see for
instance \cite{BCT05,BCT06,CTPE}. On the other hand, this scaling
corresponds to the usual normalization for the distribution function
of the sum of  independent random variables with common distribution
function of fixed variance, to obtain the central limit theorem (see
\cite{Tos05} for more comments about this analogy and relations to
the heat equation). Moreover, since in the original variables the
second moment of the solution to equation \fer{poro} diverges in time, then
scaling back with the second moment will hopefully drive the
solutions to a limit from a geometric perspective.

In the field of nonlinear diffusion equations this idea  has been at the basis of the  study of the large-time behavior of a
nonlinear diffusion equation with a general nonlinearity, of type
$u_t = \Delta\Phi(u)$ \cite{CDT05}.  More recently,  the same
scaling has been successfully applied to prove oscillatory behavior
for finite-mass solutions defined in the whole space, for particular
choices of the function $\Phi$ in the diffusion  \cite{CV07}. Also, in the fast
diffusion regime, improvements based on the second moment were also
achieved in \cite{DT}. The strategy in \cite{DT} was to introduce an
implicit scaling different from \fer{cv}, which allows to handle in a
tricky way the relative entropy functional by Newton and Ralston
\cite{New,Ral}
\be\label{rel-e}
\mF_p(u|B) = \frac 1{p-1} \int_{\RR^n} \left[ u^p-B^p
-pB^{p-1}(u-B) \right]\, dx,
\ee
commonly named \emph{relative entropy}.

In the case under study, the nonlinear scaling \fer{cv} implies that
the new function $u(y,\tau)$ satisfies the nonlinear Fokker-Planck
type equation
 \be\label{FPR}
\frac{\partial u}{\partial \tau} = \frac{\Delta_y u^p}{\int_{\RR}
u^p \, dy}+ \frac n{E_0} \nabla_y \cdot (y\, u).
 \ee
Note that \fer{new-t} implies that  $\tau(0) = 0$, while $y(0) =x$. Hence, the initial data
of the Fokker-Planck equation \fer{FPR} coincide with the initial data of
the nonlinear diffusion in original variables \fer{poro}.

At a first look, equation \fer{FPR} appears more difficult to treat
than the original nonlinear diffusion, due to the presence of the
unknown denominator of the second order term. This apparent
difficulty is balanced by the advantage to have a solution which
maintains constant its moments up to the order two. Indeed, by
construction, for all times $\tau >0$ the solution of equation \fer{FPR}
satisfies
 \be\label{norm1}
\int_{\RR^n}  u(y,\tau) \, dy = 1 , \quad \int_{\RR^n} y \, u(y,\tau) \,
dy = 0, \quad \int_{\RR^n} |y|^2 \, u(y,\tau) \, dy = E_0.
 \ee
It is easy to show that the stationary solutions of \fer{FPR} are
Barenblatt solutions of type \fer{bar}. For example, if $p <1$,
these solutions can we written as
 \bee\label{bar1}
B(|y|) = \left(\frac{C_p}{A_p + |y|^2} \right)^{1/(1-p)},
 \eee
where the  positive constants $A_p$ and $C_p$ are chosen to satisfy the
normalization conditions \fer{norm1}.

Generalized Fokker-Planck equations of type \fer{FPR} has been introduced
into the physical literature in \cite{PP}, and subsequently studied in a
series of papers (cf. the references in \cite{FD}), by studying
generalized entropies and the consequent statistics. In particular, it has been highlighted in
\cite{FD} the link between the Fokker-Planck type
equation \fer{FPR} and the R\'enyi entropy of order $p$ of a probability
density $f$ \cite{DCT}, defined by
 \be\label{re}
\mR_p(f) = \frac 1{1-p} \log\left( \int_{\RR^n} f^p(y) \, dy \right).
 \ee
The R\'enyi entropy of order $1$ is defined as the limit as $p \to 1$ of
$\mR_p(f)$. It follows directly from definition \fer{re} that
 \be\label{shan}
 \lim_{p\to 1} \mR_p(f) = \mR(u) = -\int_{\RR^n}f(y) \log f(y) \, dy .
 \ee
Then, the (Shannon) entropy is identified with the R\'enyi entropy of
order 1.

Coupled to the definition of R\'enyi entropy of order $p$, is the definition of
relative R\'enyi entropy of order $p$, valid for any pair of probability density
functions $f$ and $g$. In contrast to the universal concept of relative Shannon
entropy, this definition is not unique in the case of R\'enyi entropies, even if all
definitions converge to the relative Shannon one as $p \to 1$. In this paper, we
consider as relative R\'enyi entropy of order $p$ the quantity
 \be\label{r-re}
\mH_p(f|g) = \frac 1{(p-1)}\log  \int f^p \, dy + \log \int g^p \, dy  - \frac
p{(p-1)}\log \int g^{p-1}f \, dy.
 \ee
For example, in \cite{LYZ} the relative R\'enyi entropy is defined as $\mH_p(f|g)/p$.
The relative R\'enyi entropy of order $p$ is a nonnegative quantity (cf. the proof in
\cite{LYZ}),
 and the relative Shannon entropy is obtained as the limit of $\mH_p(f|g)$,
when $p \to 1$
 \be\label{rsh}
\lim_{p\to 1} \mH_p(f|g)  = \mH(f|g) = \int_{\RR^n}f(y) \log
\frac{f(y)}{g(y)} \, dy .
 \ee
By considering the R\'enyi entropy of the solution to equation \fer{FPR}
relative to its Barenblatt stationary solution, named $B$, since both
functions satisfy the normalization conditions \fer{norm1}, the relative
R\'enyi entropy is bounded from above by the (positive) difference between
the R\'enyi entropies  of the Barenblatt and that of the solution
 \bee\label{b1}
\mH_p(u(\tau)|B) \le \mRH_p(u(\tau)) :=  \mR_p(B)- \mR_p(u(\tau)),
 \eee
where equality holds if $p\le 1$.
Moreover, we will show that $\mRH_p(u)$ satisfies a Csiszar-Kullback type inequality
\cite{csi,kull}
 \bee\label{ck}
\mRH_p(u) \ge C(p,u)\left( \int_{\RR^n} |u- B|\, dx\right)^2.
  \eee
Our main result is exactly concerned with the time decay of
$\mRH_p(u(\tau))$. We will prove that
 \be\label{dec1}
\mRH_p(u(\tau)) \le -\frac n{2} \log \left[1 -\left(1 - \exp\left\{ - \frac
2{n}\mRH_p(v_0)  \right\} \right) \exp\left\{ - \frac {2n}{E_0}\tau  \right\} \right]\,.
 \ee
The decay \fer{dec1} indicates that in the first part of the time
axis, the convergence towards the Barenblatt solution with the same
moments up to order two, is faster that the one predicted in
terms of the Newton-Ralston relative entropy \fer{rel-e}. Then, for
larger times, the rate of decay stabilizes on the same rate
predicted by the Newton-Ralston relative entropy. Reverting to the
old variables, \fer{dec1} shows that the polynomial rate of convergence of the
solution to \fer{poro} towards the self-similar Barenblatt solution
is faster by choosing the Barenblatt with the same second moment of
the solution $v(t)$.  Then, since the second moment of the solution
to \fer{poro} converges towards the second moment of the
corresponding Barenblatt \cite{Tos05}, this improvement is lost for
large times. A similar behavior was already observed in \cite{ACK} for the heat equation using
the Shannon entropy by optimizing over Gaussians with moments equal up to order two.

The study of the large--time behavior of the solution to \fer{poro},
and its convergence towards the self--similar solution \fer{bar} is
a classical problem which has been intensively studied from many
years now. When $p>1$, this behavior has been described in dimension
$1$ in \cite{Kam1, Kam2, Va1}, in a $L^\infty$-setting. These
results have been subsequently generalized to the case $n>1$
\cite{FrKa,Va}.  Among other results, it is remarked in \cite{Va}
that the assumption of finite moments is crucial to obtain
convergence with rate. In fact, it is proved that for general data
in $L^1$ no rate can exist.

In \cite{CaTo, DD, Ot}, a different approach led to new results in
the whole space. The entropy dissipation method requires that the initial data possess a
sufficiently high number of moments (typically $2 + \delta$ for
$\delta >0$), and it is based on a suitable scaling of the solution,
which allows to rephrase the problem of convergence of the solution
towards self-similarity for \fer{poro} in terms of the convergence
of the solution to a Fokker-Planck type equation towards its steady
state. The argument usually used in this context is the so-called
\emph{entropy--production} method \cite{ACMTU, AMTU}, which relies
in the study of the relative entropy functional  \fer{rel-e}.

F.~Otto made the link with gradient flows with respect to the
Wasserstein distance in \cite{Ot} allowing to work with initial data
assuming only finite kinetic energy. D. Cordero-Erausquin,
B.~Nazaret and C.~Villani gave a proof of the corresponding
Gagliardo-Nirenberg inequalities using mass transportation
techniques in \cite{CNV}. In the fast diffusion case, these results
are restricted to the range $p > \bar p$, due to the fact
that, to work with Wasserstein's distance, it is crucial to have
second moments bounded (see for instance \cite{DMC1,DMC2,DMCK}). Further
results relative to the case $p <\bar p$ have been presented in a
series of recent papers (see \cite{BBDGV-CRAS,BBDGV,BDGV,BGV})
together with a clarification of the strategy of linearization of
the relative entropies, at least from the point of view of
functional inequalities.

The paper is organized as follows. In Section \ref{sec:FP} we use
the time-dependent scaling \fer{cv} to obtain the Fokker-Planck
equation \fer{FPR}, and we describe in details its steady solutions,
together with their main properties.

The improved rate of decay is derived in Section \ref{sec:linear}, where the strategy
is presented in the linear case $p=1$, which allows for explicit computations. In this
case, the decay is given in terms of the classical relative Shannon entropy  defined
in \fer{rsh}. The analysis of the linear case outlines in a clear way the connection
between the rate of decay of the relative Shannon relative entropy and the concavity
property of Shannon entropy power with respect to the solution to the heat equation
\cite{Cos,Vil}. The consequences of the concavity of Shannon entropy power have been
recently discussed in \cite{To1}, where the connections between this property and the
logarithmic Sobolev inequality have been clarified. The main properties of R\'enyi
entropies are listed in Section \ref{sec:renyi}. In particular, it is shown that
R\'enyi entropies constitute the natural entropies for the study of the Fokker-Planck
equation \fer{FPR}. In addition, Csiszar-Kullback type inequalities are obtained both
in the fast diffusion and in the porous medium case. Last, Section \ref{sec:nl}
contains the proof of the main result.


\section{Scaled Fokker-Planck equation and steady states}
\label{sec:FP}   

The first argument we will deal on is the derivation of the
Fokker-Planck equation \fer{FPR}, which is obtained from the
nonlinear diffusion \fer{poro} under the scaling (\ref{cv}-\ref{new-t}). Elementary computations allow to obtain from equation \fer{poro} the equation satisfied by $u(y, \tau)$. It holds
 \begin{align*}
\frac{\partial v}{\partial t} = \,& -\frac n2 \left( \frac
E{E_0}\right)^{-n/2-1} \frac{dE}{dt}\frac 1{E_0} u + \left( \frac
E{E_0}\right)^{-n/2} \frac{\partial u}{\partial \tau}\tau'(t)
 \\
&-\frac 12 \left( \frac E{E_0}\right)^{-(n+3)/2}\frac{dE}{dt}\frac
1{E_0}x\cdot\nabla_y u
 \\
=\,& -\frac 12 \left( \frac E{E_0}\right)^{-n/2-1} \frac
1{E_0}\frac{dE}{dt} \nabla_y\cdot(y u) +  \left( \frac
E{E_0}\right)^{-n/2} \frac{\partial u}{\partial \tau}\tau'(t).
 \end{align*}
Moreover, since
 \[
\Delta_x v^p =  \left( \frac E{E_0}\right)^{-np/2-1} \Delta_y u^p\,,
 \]
together with \fer{ene}, we have
 \[
\frac{dE}{dt} = 2n \int v^p\, dx = 2n \left( \frac
E{E_0}\right)^{-n(p-1)/2} \int u^p \, dy\,.
 \]
Then, we obtain that the nonlinear diffusion \fer{poro} is equivalent to
 \bee\label{new-e}
 \left( \frac E{E_0}\right)^{n(p-1)/2+1}\frac{\partial u}{\partial
 \tau}\tau'(t) = \Delta_y u^p + \frac n{E_0} \nabla_y\cdot(y u) \int u^p \,
 dy.
 \eee
Last, we choose
 \be\label{ok}
 \tau'(t) = \left( \frac E{E_0}\right)^{-n(p-1)/2-1}\int u^p \,
 dy.
 \ee
 Since $\tau'(t) >0$, \fer{ok} gives a good (implicit) time-change of
 variables. We point out that
  \[
\tau'(t) = \frac 1{2n}\frac{E_0}E \frac{dE}{dt}= \frac{E_0}{2n}
\frac d{dt} \log E(v(t)),
  \]
which implies \fer{new-t}. Hence, $u(y, \tau)$ satisfies the
Fokker-Planck equation \fer{FPR}.

Let us know look for the steady states  $\bar B(y)$ of \fer{FPR} that solve
 \be\label{ss}
\nabla_y \cdot \left( \frac{\nabla_y \bar B^p}{\int \bar B^p \, dy}+ \frac n{E_0}
y\, \bar B \right) = 0 .
 \ee
In the one-dimensional case, an exhaustive description of the properties of the steady states of various generalized Fokker-Planck equations, including \fer{FPR} can be found in \cite{FD}.  In particular, from the detailed discussion of the behaviour of the stationary solutions as the exponent $p$ tends to the lower  bound  $\bar p$ ($\bar p = 1/3$ in dimension $1$), one can argue that this lower bound is of paramount importance in connection with the eventual convergence to equilibrium.

As one can check using classical entropy arguments, to satisfy equation \fer{ss} it is sufficient to choose the argument of the divergence equal to zero
 \be\label{ss0}
\frac{\nabla_y \bar B^p}{\int \bar B^p \, dy}+ \frac n{E_0} y\, \bar B = 0.
 \ee
Multiplying \fer{ss0} by $y$ and integrating, we obtain
 \[
 \frac{\int y\cdot \nabla_y \bar B^{p}}{\int \bar B^{p} \, dy}+
\frac n{E_0} \int |y|^2 \bar B(y) \, dy = 0,
 \]
or, what is the same
 \be\label{en1}
 \int|y|^2 \bar B(y) \, dy = E_0 = \int |y|^2 u(y, \tau) \, dy.
 \ee
This fixes the second moment of the steady state to $E_0$. Suppose we are given a steady state $\bar B(y)$, which has energy $E_0$ due to \fer{en1}. Then, by defining for $\sigma>0$
 \bee\label{sc2}
\bar B_\sigma(y) = \left( \frac 1{\sqrt\sigma}\right)^{n+2} \bar B\left( \frac
y{\sqrt\sigma}\right),
 \eee
we obtain a class of functions $\bar B_\sigma$ with the same energy of $\bar B$. Moreover, a
direct inspection shows that  $\bar B_\sigma$ still satisfies
\fer{ss0} for all values of $\sigma$. Hence, by choosing $\sigma$
such that
 \be\label{gi}
\sigma\int \bar B_\sigma^{p}\, dy = \frac{E_0}n,
 \ee
the corresponding steady state satisfies the relationship
 \be\label{ok1}
\sigma\int \bar B_\sigma^{p}\, dy = \frac 1n \int|y|^2 \bar B_\sigma(y) \,
dy.
 \ee
We can easily argue that it is always possible to choose $\sigma$ to
satisfy \fer{gi}, and that this choice is unique. In fact, owing to the
definition of $\bar B_\sigma$, we obtain
 \[
\phi_p(\sigma)= \sigma\int \bar B_\sigma(y)^{p}\, dy = \sigma \left( \frac
1{\sqrt\sigma}\right)^{(n+2)p - n}\int B(y)^{p}\, dy .
 \]
Then, when $p<1$, the function $ \phi_p(\sigma) $ is (strictly)
increasing starting from $\phi_p(0) = 0$. Analogous conclusion holds
for $p>1$, which generates a strictly decreasing function $\phi_p$.

Let now $B_\sigma(y)$ be the unique probability density such that
 \be\label{pd}
 \frac p{p-1} \nabla_y B_\sigma^{p-1}(y) + \frac y\sigma= 0.
 \ee
This probability density satisfies \fer{ok1}, and consequently \fer{ss0}.
Therefore, equation \fer{ss} has always a unique solution with the right
moments up to the order two. Solving equation \fer{pd} we obtain
 for $0<p<1$
 \be\label{p1}
 B_\sigma(y) = \left( C_B + \frac{1-p}p \frac{|y|^2}{2\sigma}
 \right)^{1/(p-1)},
 \ee
on $\R^n$, while,  if $p>1$
 \be\label{p2}
B_\sigma(y) = \left( C_B - \frac{p-1}p \frac{|y|^2}{2\sigma}
 \right)_+^{1/(p-1)},
 \ee
where, as usual, $f_+(y)$ denotes the positive part of the function
$f(y)$. In \fer{p1} and \fer{p2} the constants $C_B$ and $\sigma$
can be chosen in a unique way to fix the mass and the energy.

The previous reasoning shows that, for any given initial datum
$u_0(y)$ satisfying conditions \fer{cons} and of energy equal to
$E_0$, the steady state of equation \fer{FPR} with the same mass,
momentum and energy of $u_0$ can be always represented in the form
\fer{p1} when $\bar p<p<1$ (respectively \fer{p2} when $p>1$) for a
unique choice of the parameters $C_B$ and $\sigma$. For this reason,
in what follows, for any given centered probability density function
$f$ of bounded second moment, we will always denote by
$B_{\sigma(f)}$ the Barenblatt solution  of type \fer{p1} or
\fer{p2} with the same second moment of $f$. This explicit
dependence on $f$ will be dropped when there is no possibility of
confusion.

\section{Improved decay in the linear case}
\label{sec:linear}   

We will start our analysis by considering the linear case $p=1$. Thanks to mass conservation \fer{cons} the Fokker-Planck
equation \fer{FPR} reduces to the classical linear Fokker-Planck equation
 \be\label{FPL}
\frac{\partial u}{\partial \tau} = \Delta_y u + \frac n{E_0} \nabla_y (y\,
u).
 \ee
Equation \fer{FPL} is obtained from the heat equation
 \be\label{heat}
{\partial v \over \partial t} = \Delta v ,\qquad (x\in \R^n, t>0)
 \ee
by means of the scaling
  \be\label{cvl}
v(x,t) = \left( 1 + \frac{2nt}{E_0}\right)^{-n/2} u\left[x \left( 1
+ \frac{2nt}{E_0}\right)^{-1/2}, \tau(t) \right],
 \ee
where
 \be\label{new-tt}
\tau(t) = \frac{E_o}{2n} \log \left( 1 + \frac{2nt}{E_0}\right)\,,
\quad \tau(0) = 0.
 \ee
If $p=1$, in fact, the law of evolution of the second moment is explicit, and it is
given by the linear growth
$
E(v(t)) = E_0 + 2nt.
$
We remark that the presence of the coefficient in front of the divergence term in
\fer{FPL} implies that, for a given initial value that satisfies the normalization
conditions \fer{norm1}, the solution to \fer{FPL} still satisfies \fer{norm1}. As we
will see, this property has important consequences.

The steady states of the linear Fokker-Planck equation are given by Gaussian
densities. Since any solution to \fer{FPL} which departs from an initial value that
satisfies conditions \fer{norm1} still satisfies the same conditions, the steady state
is given by the Gaussian density
 \be\label{max}
M_\sigma(y) = \left(\frac 1{2\pi \sigma} \right)^{n/2} \exp \left\{-
\frac{|y|^2}{2\sigma} \right\},
 \ee
where $\sigma = E_0/n$. For the sake of simplicity, we will simply write $
M_{E_0/n} = M$.

The goal is to study the rate of convergence of the solution of the Fokker-Planck
equation \fer{FPL} towards $M(y)$. Starting from the analysis in \cite{ToFP}, the
large-time behavior of linear Fokker-Planck type equations has been studied in details
by means of entropy-production methods \cite{AMTU}. This method relies in the study of
the evolution of the relative entropy, typically the relative Shannon entropy
\fer{rsh}. In the case under examination, in reason of the fact that both the solution
and $M$ satisfy conditions \fer{norm1}, and in particular are probability densities
with the same second moment, the relative entropy coincides with the difference
between the Shannon entropies of the Gaussian density and that of the solution
 \be\label{uB}
 \mH(u(\tau)|M ) = \mR(M)- \mR(u(\tau)).
 \ee
In \fer{uB}, $\mR(u)$ is given as in \fer{shan}.
Standard computations (cf. \cite{ToFP}) show that, if $u(\tau)$ solves \fer{FPR}
 \bee\label{de1}
\frac{d}{d\tau}\left( \mR(M)- \mR(u)\right) = - \frac{d}{d\tau} \mR(u)= - \left(
\mI(u) - \frac{n^2}{E_0} \right),
 \eee
where we defined by $\mI(f)$ the Fisher information of the probability
density $f$ as 
\be\label{fis}
\mI(f) = \int_{\{f>0\}} \frac{|\nabla f|^2} f \, dy.
 \ee
Since
$
\mI(M) = \frac{n^2}{E_0},
$
it holds
 \be\label{dec11}
\frac{d}{d\tau}\left( \mR(M)- \mR(u)\right) = - \left( \mI(u) - \mI(M)  \right).
 \ee
The well-known entropy production method consists in bounding from below the
difference between the Fisher information of $u$ and that of $M$, in terms of the
relative entropy $\mR(M) - \mR(u)$. Indeed, for any given probability density $f$, and
for any given $\sigma>0$, the logarithmic Sobolev inequality \cite{ACMTU, ToFP} takes
the form
 \be\label{LS}
 \mI(f) - \mI(M_\sigma) \ge \frac 2n \mI(M_\sigma) \left( \mR(M_\sigma)
- \mR(f) \right).
 \ee
Choosing $\sigma = E_0/n$ in \fer{LS} one obtains
 \be\label{lsi}
\mI(f) - \mI(M) \ge \frac{2n}{E_0}\left( \mR(M) - \mR(f) \right).
 \ee
Using inequality \fer{lsi} into \fer{dec11}, one shows that the relative entropy decay
exponentially at the rate $2n/E_0$
 \bee\label{decl}
\mH(u(\tau)|M ) \le \mH(v_0|M )\exp\left\{ - \frac{2n}{E_0} \tau \right\}.
 \eee
A different insight into the rate of decay in relative entropy follows by an
elementary remark, which simply relies in a different way to write the right-hand side
of equation \fer{dec11}. Indeed, equation \fer{dec11} can be equivalently written as
 \be\label{dec2}
\frac{d}{d\tau}\left( \mR(M)- \mR(u)\right) = -\frac{n^2}{E_0} \left(
\frac{\mI(u)}{\mI(M)} - 1 \right).
 \ee
The quotient of the two Fisher entropies in \fer{dec2} can be bounded from below by
resorting to the so-called \emph{Isoperimetric Inequality for Entropies} \cite{DCT,
To1}. This inequality relates Fisher information of a probability density with its
Shannon entropy power, and asserts that, for any probability density $f$, and for any
$\sigma>0$ it holds
 \be\label{b5}
{\mI(f)}\mN(f) \ge {\mI(M_\sigma)}{\mN(M_\sigma)}.
 \ee
The entropy power of a probability density $f(x)$, where $x \in \RR^n$,  is defined by
 \be\label{pow}
\mN(f) = \exp\left\{ \frac 2n \mR(f) \right\}.
 \ee
Since
$
\mR(M_\sigma) = \frac n2 \log 2\pi\sigma e,
$
the entropy power of the Gaussian density $M_\sigma$ defined in \fer{max}, in any
dimension of the space variable, is proportional to $\sigma$, with $\mN(M_\sigma) =
2\pi \sigma e$. It is immediate to verify that inequality \fer{b5} implies the
logarithmic Sobolev inequality \fer{lsi} \cite{To1}. Indeed, inequality \fer{pow} can
be rewritten in the form
 \be\label{b55}
\frac{\mI(f)}{\mI(M_\sigma)} \ge \exp\left\{- \frac 2n (\mR(f)-\mR(M_\sigma))
\right\}.
 \ee
Therefore
 \[
\frac{\mI(f)}{\mI(M_\sigma)} -1 \ge \exp\left\{- \frac 2n (\mR(f)-\mR(M_\sigma))
\right\} -1 \ge\frac 2n\left( \mR(M_\sigma) - \mR(f) \right),
 \]
which gives \fer{LS}. Inequality \fer{b5} has been derived as a consequence of the
concavity of Shannon entropy power \cite{DCT}, a property first obtained in a paper by
Costa \cite{Cos} in the framework of information theory. Due to its importance in the
forthcoming analysis, let us briefly discuss this result and its physical meaning. By
definition, the entropy power of a Gaussian random variable is proportional to its
density. Consequently, if $M_{2t}$ denotes the Gaussian self-similar solution to the
heat equation \fer{heat}, its entropy power is a linear function of time. The essence
of the concavity of Shannon entropy power is that, among all possible solutions to the
heat equation, the Gaussian solution is the only one for which the entropy power is a
linear function of time. For all the other solutions, the entropy power is a concave
function of time. The original proof of the concavity property  of entropy power was
simplified by Villani \cite{Vil}, who made an interesting connection between
information and kinetic theories. It was shown in \cite{Vil} that the concavity proof
could be based on an old argument introduced by McKean \cite{McK} in connection with
Kac caricature of the Boltzmann equation. Using inequality \fer{b55} into \fer{dec2},
and choosing $\sigma = E_0/n$, one obtains
  \be\label{dec21}
\frac{d}{d\tau}\left( \mR(M)- \mR(u)\right)  \le -\frac{n^2}{E_0} \left( \exp\left\{-
\frac 2n (\mR(u)-\mR(M)) \right\} - 1 \right).
 \ee
Equation \fer{dec21} implies a better decay with respect to the decay obtained from the
classical methods based on logarithmic Sobolev inequality. In fact, the relative
Shannon entropy $h(\tau) = \mH(u(\tau)|M)$ satisfies the differential inequality
 \be\label{opt}
\frac{dh(\tau)}{dt} \le -\frac{n^2}{E_0} \left( \exp\left\{\frac 2n
h(\tau) \right\}  -1 \right).
 \ee
Equation \fer{opt} can be solved by separation of variables to give
 \bee\label{de33}
h(\tau) \le - \frac n2 \log \left[ 1 - \left( 1- \exp\left\{-\frac 2n h(0)
\right\} \right)\exp\left\{- \frac{2n}{E_0}\tau \right\}  \right].
 \eee
Hence, the relative entropy has been shown to satisfy the decay
  \be\label{de43}
\mH(u(\tau)|M )  \le  \frac n2 \log \left[ 1 - \left( 1-
\exp\left\{-\frac 2n \mH(v_0|M ) \right\} \right)\exp\left\{-
\frac{2n}{E_0}\tau \right\} \right]^{-1}.
 \ee
By comparison, one easily concludes that the decay predicted by \fer{de43} is better
than the decay found in \fer{dec2}. Indeed, for any given positive constants $\alpha$
and $\beta$, the function
 \bee\label{r1}
r_1(\tau) =  \frac n2 \log \left[ 1 - \left( 1- e^{-\frac 2n\alpha}
\right)\exp\left\{- \beta \tau \right\}  \right]^{-1}
 \eee
is always smaller than the function
$
r_2(\tau) = \alpha e^{-\beta \tau}.$
By means of the scaling \fer{cvl}-\fer{new-tt} the time decay of the relative Shannon
entropy of the solution to the Fokker-Planck equation \fer{FPL} can be translated into
the time decay of the relative entropy between the solution $v(x,t)$ of the heat
equation \fer{heat} and the self-similar Gaussian solution which at time $t=0$ has the
same second moment of $v_0(x)$. However, it is remarkable that inequality \fer{b5} can
be applied directly to the solution to the heat equation \fer{heat}, to obtain the
time decay of the relative entropy. To this extent, for a given initial datum $v_0(x)$
that satisfies conditions \fer{norm1}, we denote by $M(x,t)$ the Gaussian density,
solution to the same heat equation \fer{heat}, which corresponds to the initial value
$M_{E_0/n}$, and still satisfies conditions \fer{norm1}. Then, for all subsequent
times $t
>0$, the moments of both $v(x,t)$ and $M(x,t)$ coincide up to order
two. In the linear case, in fact, the evolution in time of the
second moment of the solution depends only on the second moment of
the initial value. Therefore, for all times $t
>0$
 \bee\label{rl1}
\mH(v(t)|M(t)) = \mR(M(t)) - \mR(v(t)).
 \eee
Let us look for the time evolution of the Shannon relative entropy power
 \bee\label{ep}
\mN(v(t)|M(t)) = \exp\left\{ \frac 2n (\mR(v(t)) - \mR(M(t)))\right\}.
 \eee
Since both $u(x,t)$ and $M(x,t)$ satisfy equation \fer{heat}
 \[
\frac{d}{dt} \mN(v(t)|M(t)) =  \frac 2n \left(\mI(v(t)) - \mI(M(t))
\right)\mN(v(t)|M(t)) =
 \]
 \be\label{dec41}
\frac 2n \mI(M(t)) \left(\frac{\mI(v(t))}{\mI(M(t))} -1
\right)\mN(v(t)|M(t)).
 \ee
Considering that $\mI(M(t)) = n^2/(E_0 + 2nt)$, and making use of
the isoperimetric inequality \fer{b5} in \fer{dec41} one obtains
 \be \label{dec42}
\frac{d}{dt} \mN(v(t)|M(t)) \ge  \frac{2n}{E_0 + 2nt} \left( 1 -
\mN(v(t)|M(t))\right).
 \ee
Inequality \fer{dec42} can be easily solved to give
 \be\label{dec43}
\mN(v(t)|M(t)) \ge 1- \left( 1- \mN(v_0|M)) \right)\frac{E_0}{E_0+ 2nt},
 \ee
or, what is the same
 \[
\mH(v(t)|M(t)) \le - \frac n2 \log\left[ 1- \left( 1- \mN(v_0|M))
\right)\frac{E_0}{E_0+ 2nt}\right].
 \]
Clearly, since $u(y, \tau)$ and $v(x,t)$ are related by the scaling
relations \fer{cvl} and \fer{new-tt}, inequalities \fer{dec43} and
\fer{de43} coincide. The previous argument shows that in the linear
case there is no necessity to scale the problem, and consequently to
study the relaxation to equilibrium for the solution to the
Fokker-Planck equation. The results can be obtained directly from
the study of the solution to the heat equation. This direct proof is
however limited to the linear case. We proved

\begin{theorem}\label{th:li}
Let $v(x,t)$ be the solution to the initial value problem for equation \eqref{heat},
where the initial value $u_0(x)$ is of bounded Shannon entropy, and satisfies
conditions \eqref{norm1}. Then, if $M(t)$ denotes the Gaussian density satisfying
conditions \eqref{norm1} at time $t=0$, the relative Shannon entropy $\mH(v(t)|M(t))$
decays to zero with time, and satisfies the bound
 \bee\label{dec44}
\mH(v(t)|M(t)) \le - \frac n2 \log\left[ 1- \left( 1- \exp\left\{-\frac 2n
\mH(v_0|M))\right\} \right)\frac{E_0}{E_0+ 2nt}\right].
 \eee
\end{theorem}

Theorem \ref{th:li} asserts that the rate of decay to equilibrium in relative entropy,
as soon as we choose the self-similar solution with the same energy of the initial
value, is bounded in time by a quantity which proportional to the inverse of the rate
of growth of the second moment. This decay improves the known results in the first
part of the time axis, while maintaining the same rate of decay for large time. This
shows that, among all self-similar solutions that represents the intermediate
asymptotics of the solution to the linear diffusion equation \fer{heat} for large
times, in the first part of the time axis, the best approximation to the solution is
furnished by the self-similar solution with the same energy.

As we shall prove in the forthcoming section, the same property holds as soon as we
are considering the time decay of the relative R\'enyi entropy.

\section{R\'enyi entropies and nonlinear diffusions}
\label{sec:renyi}   

As briefly explained in the introduction, the second moment of the
solution to equation \fer{poro} has an important role in connection
with the knowledge of the large time behavior of the solution. This
importance can be made explicit by the following argument. Let us
consider the evolution in time of the R\'enyi entropy of order $p$
along the solution of the nonlinear diffusion equation \fer{poro},
with the same exponent $p$. Integration by parts immediately leads
to
 \be\label{e-r}
 \frac {d}{dt} \mR _p(v(\cdot,t)) =  \mI_p(v(\cdot,t)), \quad t >0,
 \ee
where, for a given probability density $f$
 \be\label{fis-r}
 \mI_p(f) := \frac 1{\int_{\RR^n} f^p \, d x}
   \int_{\{f>0\}} \frac{|\nabla f^p(x)|^2}{f(x)} \, d x.
 \ee
When $p \to 1$, identity \fer{e-r} reduces to DeBruijn's identity,
which connects Shannon's entropy functional with the Fisher
information \fer{fis} via the heat equation. Since $\mI_p(f)
>0$,  identity \fer{e-r} shows that the R\'enyi entropy of the
solution to equation \fer{poro} is increasing in time. However, no
information about the possible behavior of the solution follows.

For a given probability density $f(x)$, $x \in \R^n$, we define the
dilation of $f$ by
 \bee\label{scal1}
 f(x) \to f_a(x) = {a^n} f\left( {a} x \right), \quad a >0.
 \eee
Since dilation preserves the total mass, for any given $a>0$ the function $f_a$ will remain a probability density. Note that R\'enyi entropies are such that, for all $p >0$
 \bee\label{h-scale}
\mR_p (f_a) = \mR_p(f) - n \log a.
 \eee
Analogously, we have
 \bee\label{sec}
E(f_a) = \int_{\RR^n} |x|^2 f_a(x) \, dx = \frac 1{a^2} E(f).
 \eee
Therefore, if we define
 \bee\label{ent2}
\Lambda_p(f) = \mR_p(f) - \frac n2 \log E(f).
 \eee
the functional $\Lambda$ is invariant under dilation.

Let $v(x,t)$ be a solution to equation \fer{poro}.  If we compute
the time derivative of $\Lambda_p(v(t))$, we obtain
 \be\label{der1}
\frac d{dt}\Lambda_p(v(t)) = \mI_p(v(t)) - n^2 \frac {\int
v^p(t)}{E(v(t))},
 \ee
which is a direct consequence of both identities \fer{ene} and
\fer{e-r}.

The right-hand side of \fer{der1} is nonnegative. This can be easily
shown by an argument which is often used in this type of proofs \cite{To2}, and
goes back at least to McKean \cite{McK}.  One
obtains
\begin{align}
0 \le \,&\int_{\{v>0\}}\left( \frac{\nabla v^p(x)}{v(x)} + nx \frac
{\int v^p}{E(v)}\right)^2\, \frac{v(x)}{\int v^p } \,dx  \nonumber\\
=\,&\mI_p(v) + n^2\frac {\int v^p}{E(v)^2}\int_{\R^n} |x|^2 v(x) \, dx +
\frac {2n}{E(v)} \int_{\{v>0\}} x\cdot \nabla v^p(x)\, dx
\nonumber\\
\label{pos}
=\,&\mI_p(v) + n^2 \frac {\int v^p}{E(v)} -  2n^2\frac {\int v^p}{E(v)}
= \mI_p(v) - n^2 \frac {\int v^p}{E(v)}.
\end{align}
Note that equality to zero in \fer{pos} holds if and only if, when
$v(x)>0$
 \[
\frac{\nabla v^p(x)}{v(x)} + n x \frac {\int v^p}{E(v)} = 0.
 \]
This condition can be rewritten as
 \bee\label{pos2}
\nabla\left( v^{p-1} + \frac {p-1}{2p} |x|^2\frac{n\int v^p}{E(v)}
\right) = 0
 \eee
which identifies the probability density $v(x)$ as a Barenblatt
density in $\R^n$.
The simple addition of the second moment to R\'enyi entropy, coupled with the
constraint of invariance of the new functional $\Lambda_p$ with
respect to dilation, allows to identify in a precise way the
large-time behavior of the solution to \fer{poro}.

The argument we presented is twofold. From one side, it represents a
powerful tool to study the large-time behavior of solutions to nonlinear
diffusion equations. From the other side, it allows to find inequalities
by means of solutions to these nonlinear diffusions. Indeed, we proved
that, for any probability density function $f$ with bounded second moment
 \[
\mR_p(f) - \frac n2 \log E(f) \le \mR_p(B_\sigma) - \frac n2 \log
E( B_\sigma), \quad \mbox{for all } \sigma >0\,,
 \]
where $B_\sigma$ is the unique Barenblatt probability density satisfying \fer{pd}.
We remark that for $p >\bar p = n/(n+2)$ the quantity $\mR_p( B_\sigma) - \frac n2 \log
E( B_\sigma)$ is bounded. This is obvious if $p>1$. On the other hand, whenever $\bar p < p < 1$ the second moment of the Barenblatt \fer{bar} is bounded. Moreover, thanks to the identity
 \[
 B_\sigma^p(x) =  B_\sigma^{p-1}(x)  B_\sigma(x)\leq D (1+|x|^2)  B_\sigma(x),
 \]
where $D$ is a suitable constant, the $L^p$-norm of the Barenblatt function can be bounded in terms of the mass and the second moment.
The previous inequality  has an interesting consequence, we will use extensively in the following. Since the functional $\Lambda(f)$ is invariant with respect to dilation, the value of the right-hand side will be invariant with respect to the change of the second moment of the Barenblatt solution. Indeed, the change of the second moment of $B_\sigma$ can be obtained by dilation. By fixing the second moment of the Barenblatt equal to the second moment of $f$, one shows that
$
\mR_p( B_{\sigma(f)}) - \mR_p(f) \ge 0.
$
These results are collected next.

 \begin{lemma} Let $p > \bar p$. Then any probability density $f$ with bounded second moment satisfies the inequality
  \bee\label{in12}
\mR_p(f) - \frac n2 \log E(f) \le \mR_p(B) - \frac n2 \log
E(B_\sigma),  \quad \mbox{for all } \sigma >0\,,
 \eee
 where $B_\sigma$ is the unique Barenblatt probability density given by \eqref{pd}. Moreover, denoting by $B_{\sigma(f)}$ the Barenblatt probability density with the same second moment of $f$, then
 \be\label{posi}
\mR_p(B_{\sigma(f)}) - \mR_p(f) \ge 0.
 \ee
 \end{lemma}

Let us define the nonnegative quantity in \fer{posi} as the relative R\'enyi entropy
 \be\label{re-r}
 \mRH_p(f) = \mR_p(B_{\sigma(f)}) - \mR_p(f) = \frac1{p-1} \log \left(\frac{\int_{\RR^n} f(x)^p \,dx}{\int_{\RR^n} B_{\sigma(f)}(x)^p \,dx}\right).
 \ee
Since the second moments of $f$ and $B_{\sigma(f)}$ are equal, we can write the relative R\'enyi entropy as
\begin{align}
 \mRH_p(f) = \,&\frac1{p-1} \log \left(\frac{\int_{\RR^n} f(x)^p \,dx}{\left(\int_{\RR^n} B_{\sigma(f)}(x)^p \,dx\right)^{1-p} \left(\int_{\RR^n} B_{\sigma(f)}(x)^{p-1}f(x) \,dx\right)^{p}}\right)\nonumber\\
=\,&\mH_p(f|B_{\sigma(f)})\,,\label{re-r2}
\end{align}
for $\bar p<p<1$, due to \eqref{r-re}.

The relative R\'enyi entropy can be easily related to the relative entropy functional \fer{rel-e}
of Newton and Ralston.

\begin{lemma} Let $p > \bar p$. Then, provided the probability density $f$ and the Barenblatt solution satisfy conditions \eqref{norm1}, the relative Ralston entropy \eqref{rel-e} can be bounded from above in terms of the relative  R\'enyi entropy \eqref{re-r}, and the following bounds hold. If $p < 1$
 \be\label{pl1}
 \mF_p(f|B_{\sigma(f)}) \le \left(\int_{\RR^n} B_{\sigma(f)}^p \, dx \right) \mRH_p(f),
 \ee
 while, if $p>1$
 \be\label{pg1}
\mF_p(f|B_{\sigma(f)}) \le \left(\int_{\RR^n} f^p\, dx\right) \mRH_p(f) .
 \ee
\end{lemma}

\begin{proof}
To simplify the notation, we will write $B$ instead of $B_{\sigma(f)}$ in this proof.
Consider first the case $p<1$. Since the second moments of $f$ and $B$ coincide, \fer{posi} implies that $\int B^p \ge \int f^p$. Moreover, since $B^{p-1}$ has the form $A+
D|x|^2$, with $A$ and $D$ constants,
 \[
\mF_p(f|B) = \frac 1{p-1} \int_{\RR^n}( f^p - B^p) \, dx
 \]
 Therefore it holds
 \begin{align}
\mRH_p(f) = &\,\frac 1{1-p}\,  \log \frac{\int B^p}{\int f^p} = -\frac
1{1-p}\, \log \left( 1- \frac{ \int B^p - \int f^p  }{\int
B^p}\right) \nonumber\\
\ge &\, \frac 1{1-p}\, \frac{ \int B^p - \int f^p  }{\int B^p} =
\frac 1{\int B^p} \, \mF_p(f|B).\label{b-re}
\end{align}
In \fer{b-re} we used the inequality $-\log(1-r) \ge r$, with
 \[
 0 \le r = \frac{ \int B^p - \int f^p  }{\int B^p} \le 1.
 \]
Let now $p>1$. In this case,  \fer{posi} implies $\int B^p \le \int f^p$.
Proceeding as before, we obtain
 \be\label{b1-re}
\mRH_p(f) = \frac 1{1-p}\,  \log \frac{\int B^p}{\int f^p} \ge
\frac 1{p-1}\, \frac{ \int f^p - \int B^p  }{\int f^p}.
 \ee
On the other hand, when $p >1$, and $f$ and $B$ have the same
second moment
 \begin{align*}
\mF_p(f|B) = \,&\frac 1{p-1} \int_{\RR^n} \left[ f^p-B^p
-pB^{p-1}(f-B) \right]\, dx \\
=\,&
\frac 1{p-1} \int_{\RR^n} \left[ f^p-B^p\right] \, dx - \frac
p{p-1} \int_{|x|^2
>A} C^{p-1}(A-|x|^2)f \, dx \\
\le\,& \left(\int_{\RR^n} f^p\, dx\right) \mRH_p(f)\,,
 \end{align*}
due to \eqref{p2} and \eqref{b1-re}.
\end{proof}

\begin{remark} In consequence of inequalities \eqref{pl1} and \eqref{pg1}, all inequalities
of Csiszar-Kullback type involving the relative Ralston entropy as in \cite{CJMTU} remain
valid for the relative R\'enyi entropy.
\end{remark}

The deep connection between R\'enyi entropies and the nonlinear
diffusion equations \fer{poro}, has been recently outlined in
\cite{ST, To2}, where nonlinear diffusion equations have been introduced
as useful instruments to get inequalities in sharp form. The
analysis of \cite{ST,To2} extends to the nonlinear setting analogous
results proven in the linear case \cite{To1}.

The main idea in \cite{ST} was to consider a generalization of Shannon
entropy power \fer{pow} for R\'enyi entropies. The R\'enyi entropy power of
order $p$ of a probability density $f(x)$, where $x \in \R^n$,  has been
defined in \cite{ST} by
 \be\label{powr}
\mN_p(f) = \exp\left\{ \left( \frac 2n + p-1 \right) \mR_p(f) \right\}.
 \ee
With this definition, one verifies that the R\'enyi entropy power of
the self-similar Barenblatt solution of the nonlinear diffusion equation \fer{poro} grows linearly
with respect to time. As for Shannon entropy power, it was proven in
\cite{ST} that this linear growth is restricted to Barenblatt solutions,
while for any other solution of \fer{poro} the R\'enyi entropy power \fer{powr} is
a concave function of time. This property implies the analogous of the isoperimetric
inequality for (Shannon) entropies \fer{b5}.

\begin{theorem}\label{is}\,{\rm(\cite{ST})}
 Let $p > \bar p$. Then, for any probability density $f$, and
for any $\sigma>0$ it holds
 \be\label{iso-renyi}
\frac{\mI_p(f)}{\mI_p(B_\sigma)} \ge \exp\left\{- \left(\frac 2n
+p-1\right) (\mR_p(f)-\mR_p(B_\sigma)) \right\},
 \ee
where $B_\sigma$ denotes the unique Barenblatt probability density of type \eqref{p1} if $p<1$ or
\eqref{p2} if $p>1$ with the same second moment as $f$.
\end{theorem}

The result of Theorem \ref{is} allows to extend the result of Section \ref{sec:linear}
to the nonlinear case.


\section{Improved decay for nonlinear diffusions}
\label{sec:nl}   

We consider now the evolution of the relative R\'enyi entropy \fer{re-r} of the
solution of the nonlinear Fokker-Planck equation \fer{FPR}. Proceeding as in Section
\ref{sec:linear}, and using formula \fer{e-r} that connects R\'enyi entropy with the
generalized Fisher information \fer{fis-r}, we obtain
 \be\label{de15}
\frac{d}{d\tau}\mRH_p(u) = - \frac{d}{d\tau} \mR_p(u)= - \frac {\mI_p(u)}{\int u^p}
 + \frac{n^2}{E_0}\,,
 \ee
since $B_\sigma$ does not depend on $\tau$.
Now, consider that, owing to definition \fer{p2},  the Barenblatt solution satisfies
 \[
 \nabla B_\sigma^p = \frac p{p-1} B_\sigma \nabla B_\sigma^{p-1} = -\frac
 y\sigma B_\sigma.
 \]
Therefore, by identity \fer{ok1}
 \bee\label{fr}
 \mI_p(B_\sigma) = \frac 1{\int B_\sigma^p}\int |y|^2 B_\sigma(y) \,dy = \frac
 n\sigma.
 \eee
Last, using the expression of $\sigma$, as given by formula \fer{gi}, one concludes
with the identity
 \be\label{fin22}
\frac{n^2}{E_0} = \frac{ \mI_p(B_\sigma)}{\int B_\sigma^p}.
 \ee
Making use of \fer{fin22} into \fer{de15} we finally obtain that the time decay of the
relative R\'enyi entropy is given in the form
 \bee\label{de16}
\frac{d}{d\tau}\mRH_p(u) =  - \left(\frac {\mI_p(u)}{\int u^p} - \frac{
\mI_p(B_\sigma)}{\int B_\sigma^p} \right) = - \frac{n^2}{E_0}\left(\frac
{\mI_p(u)}{ \mI_p(B_\sigma)} \frac{\int B_\sigma^p}{\int u^p} -1 \right).
 \eee
By definition of R\'enyi entropy, one has
 \[
\frac{\int B_\sigma^p}{\int u^p} = \exp\left\{\left(p-1\right)
(\mR_p(f)-\mR_p(B_\sigma)) \right\}.
 \]
Finally, formula \fer{iso-renyi} shows that
 \bee\label{de18}
\frac{d}{d\tau}\mRH_p(u) = - \frac{n^2}{E_0}\left(\exp\left\{\frac 2n \mRH_p(u)
\right\}  -1 \right).
 \eee
Except for the meaning of the unknown quantity, the differential inequality is
identical to \fer{dec21}, relative to the linear case. Consequently, its solution
satisfies the bound
  \bee\label{de53}
\mRH_p(u(\tau))  \le  - \frac n2 \log \left[ 1 - \left( 1- \exp\left\{-\frac 2n \mRH_p(v_0)
\right\} \right)\exp\left\{- \frac{2n}{E_0}\tau \right\} \right].
 \eee
Considering that the relative r\'enyi entropy is invariant with respect to dilation,
reverting to the original variables, \fer{new-t} implies the following result:

\begin{theorem}\label{th:nl}
Let $p > \bar p$, and let $v(x,t)$ be the solution to the initial value problem for
the nonlinear diffusion equation \eqref{poro}, where the initial value $v_0(x)$ is of
bounded R\'enyi entropy, and satisfies conditions \eqref{norm1}. Then, if
$B_{\sigma(v_0)}$ denotes the Barenblatt density satisfying conditions \eqref{norm1},
the relative R\'enyi entropy $\mRH_p(v(t))$ defined in \eqref{re-r} decays to zero with
time, and satisfies the bound
 \bee\label{dec44}
\mRH_p(v(t)) \le - \frac n2 \log\left[ 1- \left( 1- \exp\left\{-\frac 2n
\mRH_p(v_0)\right\} \right)\frac{E_0}{E(v(t))}\right].
 \eee
\end{theorem}

\section{Final remarks}

To finish this work, let us make some remarks about the connections of the scaled equation \eqref{FPR}, the relative R\'enyi entropy functional $\mRH_p(u(\tau))$, and gradient flow structures. We know from \cite{Ot} that the original equation in self-similar variables \eqref{poro} has the structure of gradient flow of the Ralston-Newman entropy functional $\mF_p(f|B_\sigma)$ with respect to the Euclidean Wasserstein distance. This structure is also kept for the scaled equation \eqref{FPR}. Indeed, in the parameter interval $\bar p<p<1$ we can  check that equation \eqref{FPR} is equivalent to
$$
\frac{\partial u}{\partial \tau}=\nabla_y\cdot \left(u \nabla_y \frac{\delta \mRH_p(u(\tau)|B_\sigma)}{\delta u}\right).
$$
This is immediate. By \eqref{re-r2}, we get
$$
\frac{\delta\mRH_p(f|B_\sigma)}{\delta f} = \frac{p}{p-1} \frac{f^{p-1}}{\int_{\RR^n} f^p\, dx}
                                           -\frac{p}{p-1} \frac{B_\sigma^{p-1}}{\int_{\RR^n} f B_\sigma^{p-1}\, dx}\,.
$$
Since $B_\sigma$ has the same mass and second moment of $f$, and $B_\sigma^{p-1}=A+D|x|^2$, with $A$ and $D$ constants, then for $p<1$ we get
$$
\frac{\delta \mRH_p(f|B_\sigma)}{\delta f} = \frac{p}{p-1} \frac{f^{p-1}}{\int_{\RR^n} f^p\, dx}
                                           -\frac{p}{p-1} \frac{B_\sigma^{p-1}}{\int_{\RR^n} B_\sigma^{p}\, dx}\,,
$$
that together with \eqref{ok1} and \eqref{pd} leads to
$$
\frac{\delta \mRH_p(f|B_\sigma)}{\delta f} = \frac{p}{p-1} \frac{f^{p-1}}{\int_{\RR^n} f^p\, dx}
                                           +\frac{n}{2E_0} |y|^2\,.
$$
Therefore, the scaled equation \eqref{FPR} is the formal gradient flow of the relative R\'enyi entropy functional $\mRH_p(f|B_\sigma)$ with respect to the Euclidean Wasserstein distance for $\bar p<p<1$. Moreover, due to the formula \eqref{re-r}, we can write the relative R\'enyi entropy functional $\mRH_p(f)$ as
$$
\mRH_p(f|B_\sigma) = \frac1{1-p} \log \left(\frac{\int_{\RR^n} B_{\sigma}(x)^p \,dx}{(p-1)\mathcal{E}_p(f)}\right)\,,
\mbox{ with }
\mathcal{E}_p(f)=\frac{1}{p-1} \int_{\RR^n} f(x)^p \,dx\,.
$$
We know from the by now classical McCann's condition, that the functional $\mathcal{E}_p(f)$ is displacement convex for $(n-1)/{n}\leq p <1$. Therefore, it is trivial to check that the relative R\'enyi entropy functional $\mRH_p(f|B_\sigma)$ is displacement convex in the same range, provided $p>\bar p$ for being well-defined, since it is the composition of convex functions. Therefore, all the standard theory of gradient flows in probability measures of geodesically convex functionals, as in \cite{AGS}, applies directly to the scaled equation \eqref{FPR} in the range $\max(\frac{n-1}{n},\frac{n}{n+2})<p<1$. This breaks down for $p>1$ since \eqref{re-r2} is no longer true.

The previous discussion, coupled with the results of this work, highlight a strong and new connection between nonlinear Fokker-Planck type equations and R\'enyi entropies. In particular, they furnish a mathematical structure to the statistical framework introduced in \cite{FD} to introduce and justify this type of equations.


\bigskip
\noindent {\bf Acknowledgments:} JAC acknowledges support from projects MTM2011-27739-C04-02,
2009-SGR-345 from Ag\`encia de Gesti\'o d'Ajuts Universitaris i de Recerca-Generalitat de Catalunya, the Royal Society through a Wolfson Research Merit Award, and the Engineering and Physical Sciences Research Council (UK) grant number EP/K008404/1. GT acknowledges support from MIUR project ``Optimal mass transportation, geometrical and functional inequalities with applications'', and from the GNFM group of National Institute of High Mathematics of Italy (INDAM).


\end{document}